\documentclass[12pt,a4paper]{amsart}
\usepackage{amssymb}
\usepackage{amscd}
\usepackage[backref=page]{hyperref}

\usepackage{amsmath}
\usepackage{amsthm}
\usepackage{amsfonts}
\usepackage{amssymb}
\usepackage{euscript}
\usepackage[all]{xy}
\usepackage{tikz}
\usepackage{tikz-cd}

\newtheorem{theoremA}{Theorem}

\tikzcdset{scale cd/.style={every label/.append style={scale=#1},
		cells={nodes={scale=#1}}}}

\newcommand{\nc}{\newcommand}

\nc{\la}{\lambda}
\nc{\al}{\alpha }
\nc{\om}{\omega }

\nc{\veps}{\varepsilon}
\nc{\ch}{{\mathop {\rm ch}}}
\nc{\Tr}{{\mathop {\rm Tr}\,}}
\nc{\Id}{{\mathop {\rm Id}}}
\nc{\bra}{\langle}
\nc{\ket}{\rangle}
\nc{\pa}{\partial}
\nc{\ld}{\ldots}
\nc{\cd}{\cdots}
\nc{\hk}{\hookrightarrow}
\nc{\T}{\otimes}
\nc{\mgl}{\mathfrak{gl}}
\nc{\U}{\mathrm U}

\numberwithin{equation}{section}
\newtheorem{thm}{Theorem}[section]
\newtheorem{prop}[thm]{Proposition}
\newtheorem{lem}[thm]{Lemma}

\theoremstyle{remark}

\newtheorem{conj}[thm]{Conjecture}

\title{From brick manifolds to Grassmannians of bimodules}

\author{Evgeny Feigin}
\address{Evgeny Feigin:\newline
School of Mathematical Sciences, Tel Aviv University, Tel Aviv
6997801, Israel}
\email{evgfeig@gmail.com}

\author{Markus Reineke}
\address{Markus Reineke:\newline
Ruhr-Universit\"at Bochum, Faculty of Mathematics, Universit\"atsstra{\ss}e 150, 44780 Bochum, Germany}
\email{Markus.Reineke@ruhr-uni-bochum.de}

\begin{document}
\begin{abstract}
We study a class of Grassmannians of sub-bimodules over the path algebras 
of quivers. Our quiver Grassmannians include Escobar's brick manifolds as well 
as Labelle's generalizations. We give an explicit construction of the varieties
in question, provide examples and clarify connection with the quiver representation
spaces. We also prove smoothness of our Grassmannians, construct cellular
decompositions and derive a realization as framed  moduli spaces. The framed moduli  
realization leads to a recursive formula for the motives.
\end{abstract}

\maketitle

\section{Introduction}
The brick manifolds introduced in \cite{E16} are certain fibers of the Bott-Samelson
varieties which are not birationally isomorphic to the flag varieties. They turn out 
to be useful in various situations, see e.g. \cite{DSV19,CGGS21,Sp23}. In type $A$
case under certain conditions the brick manifolds are smooth toric varieties admitting
an explicit realization as collections of subspaces in an ambient vector space satisfying
a number of incidence conditions.

A generalization of the type $A$ brick manifolds was recently suggested by Labelle in 
\cite{L25}. The generalized brick manifolds are no longer toric, but still smooth and
admit an explicit realization in terms of collections of subspaces. Conjecturally, their
Poincar\'e polynomials are equal to a specialization of Toda eigenfunctions 
\cite{BF06,GL03}. The conjecture is equivalent to a certain fermionic type recursion 
for the Poincar\'e polynomials \cite{BF14,FFJMM09,L25}. 

All the above mentioned constructions have to do with the equioriented type $A$
quiver. More precisely, the generalized brick varieties can be realized as Grassmannians
of bimodules of the corresponding path algebras. 
The goal of this paper is to extend the setup to the case of arbitrary
quivers with no parallel paths and to study the varieties thus obtained. 
Below we provide some details
of the construction and formulate the main results obtained in our paper.

Let $Q$ be an acyclic quiver without parallel paths; let $Q_0$ be the set of vertices and 
$Q_1$ be the set of edges. We fix a collection of vector spaces $V_*=(V_i)_{i\in Q_0}$ 
and for each pair
$(i,j)\in Q_0^2$ of vertices of $Q$ we define the vector space
\[
M(V_*)_{i,j} = \bigoplus_{\omega,\omega'} V_p\quad \text{ such that }\quad
j\stackrel{\omega'}{\leadsto}p\stackrel{\omega}{\leadsto}i
\]
(i.e. $\omega$ is a path in $Q$ from a vertex $p$ to $i$ and $\omega'$ is a path in $Q$ from $j$ to $p$).
The collection $M(V_*)=(M(V_*)_{i,j})_{i,j}$ has a natural structure of a bimodule
over the path algebra $A$ of the quiver $Q$; in other words, $M(V_*)$ carries a natural structure
of a module over $A\otimes A^{\rm op}$. Our main objects of study are the varieties
$$X(V_*) = {\rm Gr}^{{\bf f}(V_*)}_{A\otimes A^{\rm op}}(M(V_*))$$
of sub-bimodules of $M(V_*)$ of certain codimension  ${\bf f}(V_*)=({\bf f}_{i,j})_{i,j}$.
The codimensions ${\bf f}_{i,j}$ vanish unless there exists a path from $j$ to $i$ in $Q$;
if such a path does exist, then ${\bf f}_{i,j}=\dim V_i$. These varieties can be also 
realized as quiver Grassmannians for certain quivers (see also \cite{CFR12,CR08,CI20}). 

One observes that if $Q$ is an equioriented type $A$ quiver, then the Grassmannians
above are isomorphic to the generalized brick manifolds from \cite{L25}. 
We prove the following theorem, where ${\bf d}=(\dim V_i)_{i\in Q_0}$ and 
$G_{\bf d}=\prod_{i\in Q_0} GL_{d_i}$.

\begin{theoremA}
Assume that $Q$ is acyclic and has no parallel paths. 
Then $X(V_*)$ is a smooth $G_{\bf d}$-equivariant 
compactification of the  representation space $R_{\bf d}(Q)$. 
\end{theoremA}

Our next task is to study the motives $[X(V_*)]$ of the varieties $X(V_*)$. More precisely,
we construct a cellular decomposition of our Grasmannians of bimodules and derive a
recursive formula for their motives $[X(V_*)]$ which involves more general Grassmannians.
We note that our formula  in the equioriented type $A$ case is different from the one
conjectured in \cite{L25}.  

\begin{theoremA}
Let $Q$ be an acyclic quiver with no parallel paths. Then $X(V_*)$ admits a cellular 
decomposition with the cells labeled
by the torus fixed points and the motives $[X(V_*)]$ satisfy the recursive formula 
from Theorem \ref{thm:recursion}.
\end{theoremA}

Let us also note that our construction works for quivers with parallel paths as well.
However, the varieties thus constructed are more complicated; in particular, they are
no longer smooth in general. We provide the definitions and examples at the end of our
paper.

The paper is organized as follows. In section \ref{sec:quiversandbimodules} we introduce basic notation on quivers and bimodules over path algebras.
In section \ref{sec:sub-bimodules} we introduce our main objects of study -- certain quiver Grassmannians of sub-bimodules.
In section \ref{sec:examples} we provide several examples of the general construction.
In section \ref{sec:compactification} we show that our Grassmannians compactify representation 
spaces and in section \ref{sec:smoothness} we give a proof of smoothness.
In section \ref{sec:framedmoduli} we provide a framed moduli realization and derive a recursive formula for the motives of the framed moduli spaces. 
Finally, in section \ref{sec:parallel} we generalize our construction to the case of quivers with parallel paths.

\subsection*{Acknowledgements} 
E.F. was partially supported by the ISF grant 493/24.

\section{Quiver representations and bimodules}\label{sec:quiversandbimodules}
We start with introducing basic definitions from the theory of quivers (see e.g. \cite{ASS06,Sch14}).
Let $k$  be an algebraically closed field, and let $Q$ be an acyclic quiver with set of vertices $Q_0$ 
and arrows written as $\alpha:i\rightarrow j$. We denote by $A=kQ$ the (finite-dimensional) path algebra 
of $Q$ over $k$. The unit of the path algebra decomposes into orthogonal idempotents
$1=\sum_{i\in Q_0}e_i$. Then $e_iAe_j$ is linearly spanned by the paths from $j$ to $i$.

Left modules over $A$ correspond to representations of $Q$ over $k$; namely, to a representation
$V=((V_i)_i,(V_\alpha)_\alpha)$ we associate the module with underlying $k$-vector space
$V=\bigoplus_iV_i$, the generator $e_i$ of $A$ acting as the projection to the $V_i$-component of $V$,
and the generator $\alpha$ acting via projecting from $V$ to the $V_i$-component, 
mapping with $V_\alpha$ to $V_j$, and embedding $V_j$ into $V$. Conversely, given a left module 
$V$ over $A$, we define $V_i=e_iV$ for $i\in Q_0$ and $V_\alpha$ as the action of $\alpha$ on $V$, 
which gives a well-defined map from $V_i$ to $V_j$, for all $\alpha:i\rightarrow j$.

In particular, the indecomposable projective representation $P_i$ corresponds to $Ae_i$ and the 
indecomposable injective representation $I_i$ corresponds to $A^*e_i$ 
(where $A^*={\rm Hom}_k(A,k)$ denotes the linear dual). Every left $A$-module $V$ then 
admits a standard projective resolution
$$0\rightarrow\bigoplus_{\alpha:i\rightarrow j}Ae_j\otimes_ke_iV\stackrel{\alpha}{\rightarrow}\bigoplus_iAe_i\otimes_ke_iV\stackrel{\beta}{\rightarrow} V\rightarrow 0,$$
where the maps are given as follows: for a path $\omega$ in $Q$ starting in $i$ (which is an element 
of $Ae_i$) and an element $v$ of $e_iV$, we define $\beta(\omega\otimes v)=V_\omega(v)$. 
Here $V_\omega$ is defined as $V_{\alpha_s}\circ\ldots\circ V_{\alpha_1}$ if 
$\omega=\alpha_s\ldots\alpha_1$. For a path $\omega$ in $Q$ starting in $j$ and an element 
$v$ of $e_iV$, we define $\alpha(\omega\otimes v)=\omega\alpha\otimes v-\omega\otimes V_\alpha(v)$.

Now we consider $A$-$A$-bimodules. By definition, these are equivalent to left 
$A\otimes A^{\rm op}$-modules. We have $(kQ)^{\rm op}=k(Q^{\rm op})$, where $Q^{\rm op}$ denotes 
the opposite quiver, with an arrow $\alpha^*:j\rightarrow i$ for every arrow $\alpha:i\rightarrow j$ 
in $Q$. It is then easy to see that 
\begin{equation}\label{eq:ideal}A\otimes A^{\rm op}\simeq k(Q\times Q^{\rm op})/I,
\end{equation}
where $Q\times Q^{\rm op}$ is the quiver with vertices $(i,j)$ for $i,j\in Q_0$ and arrows 
$(\alpha,j):(i,j)\rightarrow (i',j)$ for all $\alpha:i\rightarrow i'$ in $Q$ and $j\in Q_0$, 
as well as arrows $(i,\alpha^*):(i,j')\rightarrow (i,j)$ for all 
$i\in Q_0$ and $\alpha:j\rightarrow j'$ in $Q$. The ideal $I$ is generated by all commutativity
relations, that is, by all  $$(\alpha,j)(i,\beta^*)-(i',\beta^*)(\alpha,j')$$
for $\alpha:i\rightarrow i'$ and $\beta:j\rightarrow j'$ in $Q$. We can thus view $A$-$A$-bimodules 
$M$ as representations of $Q\times Q^{\rm op}$, given by $k$-vector spaces $M_{i,j}$ for 
$i,j\in Q_0$ and linear maps $M_{\alpha,j}:M_{i,j}\rightarrow M_{i',j}$ for all 
$\alpha:i\rightarrow i'$ in $Q$ and $j\in Q_0$, as well as linear maps 
$M_{i,\alpha^*}:M_{i,j'}\rightarrow M_{i,j}$ for all $i\in Q_0$ and $\alpha:j\rightarrow j'$ in $Q$, such that
$$M_{\alpha,j}\circ M_{i,\beta^*}=M_{i',\beta^*}\circ M_{\alpha,j'}$$
for $\alpha:i\rightarrow i'$ and $\beta:j\rightarrow j'$ in $Q$.

%To prove: global dimension of $A\otimes A^{\rm op}$ is at most two, and injective dimension of projectives is at most $1$.

The regular $A$-$A$-bimodule $A\otimes_kA^*$ decomposes into indecomposable projective bimodules 
$P_{ij}=Ae_i\otimes e_jA$ for $i,j\in Q_0$. We have
$${\rm Hom}_{A\otimes A^{\rm op}}(P_{ij},P_{kl})\simeq e_iAe_k\otimes e_kAe_k$$
for all $i,j,k,l\in Q_0$.

\section{A class of Grassmannians of sub-bimodules}\label{sec:sub-bimodules}
From now on and till the last section we assume that $Q$ is a tree, that is, there are no parallel paths 
in $Q$. For every $i\in Q_0$ we fix a $k$-vector space $V_i$, and consider the bimodule
$Ae_i\otimes_kV_i\otimes_ke_iA$ (which is isomorphic to $(\dim V_i)$-many copies of the indecomposable
projective bimodule $P_{ii}$). 
Finally, we define
\begin{equation}
M(V_*)=\bigoplus_{i\in Q_0}Ae_i\otimes V_i\otimes e_iA.
\end{equation}
To describe this as a representation of $Q\times Q^{\rm op}$ fulfilling the commutativity relations, 
we calculate
$$M(V_*)_{i,j}=e_iM(V_*)e_j=\bigoplus_pe_iAe_p\otimes V_p\otimes e_pAe_j=\bigoplus_{(\omega,\omega')}V_p,$$
where the final direct sum ranges over all pairs of paths $(\omega,\omega')$ in $Q$ such that
$$j\stackrel{\omega'}{\leadsto}p\stackrel{\omega}{\leadsto}i.$$
All maps are given by obvious inclusions; we only treat the case of the map $M(V_*)_{\alpha,j}$ for 
an arrow $\alpha:i\rightarrow i'$ in $Q$: 
given a component $V_p$ of $M(V_*)_{i,j}$ corresponding to a pair of paths 
$j\stackrel{\omega'}{\leadsto}p\stackrel{\omega}{\leadsto}i$, we have a pair of paths 
$j\stackrel{\omega'}{\leadsto}p\stackrel{\alpha\omega}{\leadsto}i'$, yielding a corresponding 
component $V_p$ in $M(V_*)_{i',j}$.

Let $\overline Q_0\subset Q_0\times Q_0$ be the set of pairs of vertices $(i,j)$ such that there exists a path from 
$j$ to $i$ in $Q$. Then the target of  any arrow in $Q\times Q^{op}$ whose source
belongs to $\overline Q_0$ is also in $\overline Q_0$. We denote by $\overline Q$ the full subquiver
of $Q\times Q^{op}$ with the set of vertices $\overline Q_0$. It follows from the 
definition that $M(V_*)$ is completely determined by its restriction to 
$\overline Q$.

There is a natural structure of a poset on the set $\overline Q_0$. Namely, given
two pairs pair $(i_1,j_1),(i_2,j_2)\in Q_0^2$ connected by  paths $j_1\stackrel{\omega_1}{\leadsto}i_1$ and $j_2\stackrel{\omega_2}{\leadsto}i_2$ 
we write $(i_1,j_1)\le (i_2,j_2)$ if $\omega_1$ is a subpath of $\omega_2$.

Recall that an upper ideal is a subset $C$ such that $c_1\in C$ implies that $c_2\in C$ for 
any $c_2\ge c_1$. We call an upper ideal indecomposable if it can not be written as a disjoint
union of two upper ideals.

\begin{lem}
There is one-to-one correspondence between the indecomposable subrepresentations of 
$M(V_*)$ and the indecomposable upper ideals of the poset $\overline Q_0$.
\end{lem}
\begin{proof}
Let $C\subset \overline Q_0$ be an indecomposable upper ideal. 
We define the $\overline Q$ module $M(C)$ with $\dim M(C)_c=1$
for $c\in C$ and $M(C)_c=0$ for $c\notin C$ in an obvious way, i.e. with all arrows
between non-zero spaces being isomorphisms. Then $M(C)$ is obviously an indecomposable subrepresentation of $M(V_*)$.

Now assume we are given an indecomposable $\overline Q$ module $M\subset M(V_*)$. Let $C\subset \overline Q_0$ be the set of  
elements $c\in Q_0$ such that $M_c\ne 0$. We claim that 
$C$ is an indecomposable upper ideal and $M\simeq M(C)$.  
First, $C$ is an upper ideal because all the maps in $M(V_*)$ are injective. 
Second, it is indecomposable, since a decomposition $C=C_1\sqcup C_2$
into a disjoint union of two upper ideals would lead to the decomposition of the module $M$.
Finally, let us show that all the non-trivial components of $M$ are one-dimensional. 
We choose a minimal element $c$ of the poset $\overline Q_0$ such that $M_c\ne 0$.
Choosing a vector $v\in M_c$ we construct a subrepresentation $N\subset M$ as follows.
First,  $v\in N$. Second, $N$ contains $N(v)$ -- a subrepresentation of $M$ generated by $v$. 
Third, we add to $N$ all vectors of $M$ which are mapped to $N(v)$ by all possible compositions
of arrows. Since all the maps of $M(V_*)$ are embeddings, all the non-trivial components of $N$
are one-dimensional. We are left to show that $N$ is a direct summand of $M$. This follows
from the injectivity of the arrows and from the assumption that $Q$ has no parallel paths. 
\end{proof}

We define a dimension vector ${\bf f}={\bf f}(V_*)$ for $Q\times Q^{\rm op}$ 
(or, equivalently, for $\overline Q$) by
%$${\bf f}_{i,j}=\left\{\begin{array}{lcr}\dim V_i&,&\mbox{ there exists a path from $j$ to $i$ in $Q$},\\ 0&,&\mbox{ otherwise}\end{array}\right. $$
$${\bf f}_{i,j}=\begin{cases}
\dim V_i,  &    \text{ there exists a path from $j$ to $i$ in $Q$},\\
0, & \text{ otherwise}.
\end{cases}
$$

Our central object of interest is the Grassmannian of quotient bimodules
$$X={\rm Gr}^{{\bf f}(V_*)}_{A\otimes A^{\rm op}}(M(V_*)).$$

\begin{lem}
Let ${\bf e}=\dim M(V_*) - {\bf f}(V_*)$. Then $X$ is isomorphic to the Grassmannian 
of $A$ sub-bimodules ($A\otimes A^{\rm op}$ modules) ${\rm Gr}_{{\bf e}}(M(V_*))$.    
\end{lem}
\begin{proof}
The desired isomorphism sends an element of ${\rm Gr}^{{\bf f}(V_*)}_{A\otimes A^{\rm op}}(M(V_*))$
to the quotient representation.
\end{proof}

\section{Examples}\label{sec:examples}
\subsection{Linearly oriented type A}
Let $Q$ be the linearly oriented type $A_n$ quiver given by 
$1\rightarrow 2\rightarrow\ldots\rightarrow n$. Then 
$$M(V_*)_{ij}=V_j\oplus\ldots\oplus V_i  \text{ for } j\leq i,$$
and zero otherwise. 
All maps in this representation are given by the natural inclusions of
direct summands. We have ${\bf f}_{i,j}=\dim V_i$ for $j\leq i$, and zero otherwise. 
We can interpret $X={\rm Gr}^{\bf f}_{A\times A^{\rm op}}(M(V_*))$ as the Grassmannian of 
sub-bimodules
$${\rm Gr}^{A\otimes A^{\rm op}}_{\bf e}(M(V_*)),$$
where ${\bf e}$ is the dimension vector of $N(V_*)$, which is given by
$${\bf e}_{i,j}=\dim V_j+\ldots+\dim V_{i-1}$$ for $j<i$, and zero otherwise. 
We thus arrive at the variety $X(d_1,\ldots,d_n)$ defined in \cite{L25} ($d_i=\dim V_i$) (if all $d_i=1$, one gets the Escobar brick manifolds \cite{E16}).

Here is a picture for the representation $M(V_*)$
for $Q=A_3$:
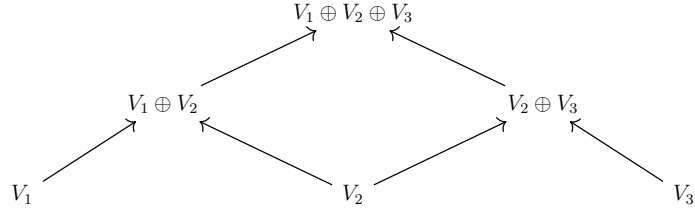
\[
\begin{tikzcd}[scale cd = 0.7]
& & V_1\oplus V_2\oplus V_3 \\
 & V_1\oplus V_2 \arrow[ur] & &  V_2\oplus V_3 \arrow[ul] \\
V_1 \arrow[ru] & & V_2 \arrow[ru]\arrow[ul] & & V_3\arrow[ul]
\end{tikzcd}
\]
The dimension vector ${\bf e}$ is given by \
\begin{tabular}{ccccc}
& & $d_1+d_2$ & &\\
& $d_1$ & & $d_2$ &\\
0 &  & 0 & & 0
\end{tabular}

\subsection{Alternating orientation quivers}
Let $Q$ be a quiver with alternating orientation, that is, every vertex is a source or a sink.
%This partitions the set of vertices $Q_0=I\cup J$ into the set $I$ of sources and the set $J$ of sinks.
We then find 
$$M(V_*)_{i,j}=\left\{\begin{array}{ccl}V_i&,&i=j,\\ V_i\oplus V_j&,&\mbox{there is an arrow }\alpha:j\rightarrow i,\\ 0&,&\mbox{otherwise}\end{array}\right.$$
The only non-zero maps in $M(V_*)$ are the following inclusions of direct summands for all 
$\alpha:i\rightarrow j$:
$$M(V_*)_{\alpha,i}:V_i\rightarrow V_i\oplus V_j,\;\;\; M(V_*)_{j,\alpha^*}:V_j\rightarrow V_i\oplus V_j.$$
In particular, there are no proper commutativity relations fulfilled by this representations.
In fact, we can view it as a representation of the quiver $\Gamma$ with vertices $i\in Q_0$ 
and $v_\alpha$ for $\alpha$ an arrow in $Q$, and arrows $i\rightarrow v_\alpha\leftarrow j$ 
for all $\alpha:i\rightarrow j$ in $Q$. This is again a quiver with alternating orientation,
and $M(V_*)$ is the projective representation $\bigoplus_{i\in Q_0}P_i\otimes V_i$. 

The non-zero components of the dimension vector ${\bf e}={\bf e}(V_*)$ are given by
${\bf e}_{i,j}=\dim V_j$ for all arrows $j\to i$ in $Q$. Hence one gets
\[
{\rm Gr}_{\bf e}(M(V_*))\simeq 
\prod_{\alpha\in Q_1} {\rm Gr}_{\dim V_{s(\alpha)}}(V_{s(\alpha)}\oplus V_{t(\alpha)}). 
\] 
%Using results of 
%\cite{R08}, they can be realized as iterated Grassmann bundles, so their Poincaré polynomials 
%in singular cohomology can be determined as products of $q$-binomial coefficients. 
%In this way, we find
%$$\sum_i\dim H^k(X)q^k=\prod_{\alpha:i\rightarrow j}\left[{d_i+d_j\atop d_j}\right]_q$$
%for $Q$ with alternating orientation.

\subsection{Type D quiver}
Let $Q$ be the $D_4$ quiver with vertices $1,2,3,4$ and arrows $1\to 2$, $2\to 3$ and $2\to 4$.
Then the quiver $\overline Q$ is of the form
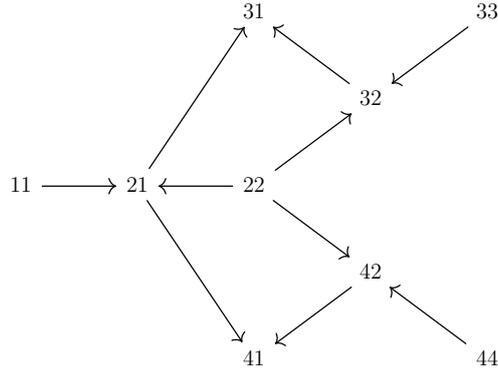
\[
\begin{tikzcd}[scale cd = 0.7]
&& 31 & & 33 \arrow[dl]\\
&&  & 32 \arrow[ul]&  \\
11 \arrow[r] & 21 \arrow[uur]\arrow[ddr] & 22\arrow[l]\arrow[ur]\arrow[dr] & & \\
&&  & 42 \arrow[dl]&  \\
&& 41 & & 44 \arrow[ul]
\end{tikzcd}	
\]
The representation $M(V_*)$ looks as follows:

\medskip

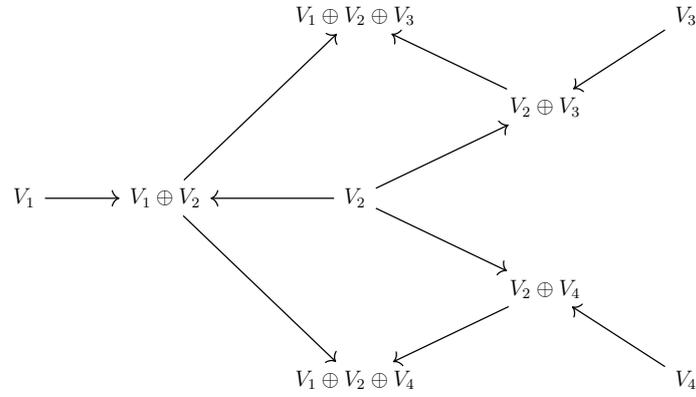
\[
\begin{tikzcd}[scale cd = 0.7]
&& V_1\oplus V_2\oplus V_3 & & V_3 \arrow[dl]\\
&&  & V_2\oplus V_3 \arrow[ul]&  \\
V_1 \arrow[r] & V_1\oplus V_2 \arrow[uur]\arrow[ddr] & V_2\arrow[l]\arrow[ur]\arrow[dr] & & \\
&&  & V_2\oplus V_4 \arrow[dl]&  \\
&& V_1\oplus V_2\oplus V_4 & & V_4 \arrow[ul]
\end{tikzcd}	
\]
and the dimension vector ${\bf e}$ is given by 
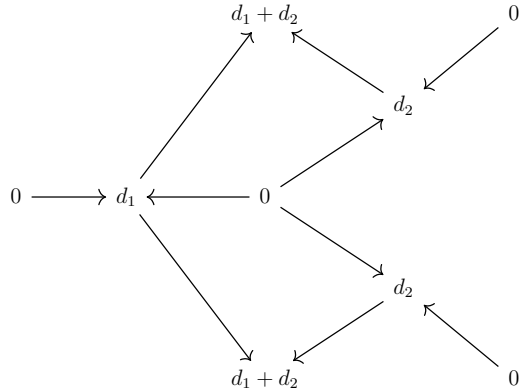
\[
\begin{tikzcd}[scale cd = 0.7]
&& d_1+d_2 & & 0 \arrow[dl]\\
&&  & d_2 \arrow[ul]&  \\
0 \arrow[r] & d_1 \arrow[uur]\arrow[ddr] & 0\arrow[l]\arrow[ur]\arrow[dr] & & \\
&&  & d_2 \arrow[dl]&  \\
&& d_1+d_2 & & 0 \arrow[ul]
\end{tikzcd}
\]
Hence the Grassmannian of subbimodules ${\rm Gr}_{\bf e}(M(V_*))$ is isomorphic to the quiver
Grassmannian for the alternating $A_5$ quiver with representation
\[
V_2\oplus V_3 \rightarrow V_1\oplus V_2\oplus V_3 \leftarrow V_1\oplus V_2 \rightarrow V_1\oplus V_2\oplus V_4 \leftarrow V_2\oplus V_4 
\]
and the dimension vector $(d_2,d_1+d_2,d_1,d_1+d_2,d_2)$.

\section{Compactification of representation spaces}\label{sec:compactification}
The goal of this section is to prove that our Grassmannians of sub-bimodules  
equivariantly compactify the representation spaces. 

Let ${\bf d}=(\dim V_i)_{i\in Q_0}$ and let 
$$R_{\bf d}(Q)=\bigoplus_{\alpha:i\rightarrow j}{\rm Hom}_k(V_i,V_j)$$
be the representation space. The group
$G_{\bf d}=\prod_{i\in Q_0}{\rm GL}(V_i)$ acts on $R_{\bf d}(Q)$ 
by the base change; the orbits of this action correspond to isomorphism classes of 
$k$-representations of $Q$ of dimension vector ${\bf d}$.

Recall the notation $X={\rm Gr}^{{\bf f}(V_*)}_{A\otimes A^{\rm op}}(M(V_*))$. 
The variety $X$ admits a natural $G_{\bf d}$ action induced by the obvious action 
on the spaces $V_i$. We prove the following theorem.

\begin{thm}
There is a $G_{\bf d}$ equivariant emebedding $R_{\bf d}(Q)\subset X$ such that 
the image is dense in $X$.
\end{thm}
\begin{proof}
We consider the projective bimodule
$$N(V_*)=\bigoplus_{\alpha:i\rightarrow j}Ae_j\otimes V_i\otimes e_iA$$
and consider bimodule maps from $N(V_*)$ to $M(V_*)$ (inspired by the form of the standard projective resolution above).
%; the precise relation has to be explored!). 
We then have
\begin{multline*}
{\rm Hom}_{A\otimes A^{\rm op}}(N(V_*),M(V_*))\simeq\\
\bigoplus_{\alpha:i\rightarrow j}\bigoplus_{p\in Q_0}{\rm Hom}_{A\otimes A^{\rm op}}(Ae_j\otimes V_i\otimes e_iA,Ae_p\otimes V_p\otimes e_pA)\simeq\\
\bigoplus_{\alpha:i\rightarrow j}\bigoplus_{p\in Q_0}{\rm Hom}_k(V_i,{\rm Hom}_{A\otimes A^{\rm op}}(Ae_j\otimes e_iA,Ae_k\otimes e_kA)\otimes V_k)\simeq\\
\simeq\bigoplus_{\alpha:i\rightarrow j}{\rm Hom}_k(V_i, V_i\oplus V_j).
\end{multline*}
Here we use the fact that $${\rm Hom}_{A\otimes A^{\rm op}}(Ae_j\otimes e_iA,Ae_k\otimes e_kA)\simeq e_jAe_k\otimes e_kAe_i$$
is spanned by $\alpha\otimes e_i$ and $e_j\otimes \alpha$, making use of the assumption that there are no parallel paths in $Q$.

Similarly, we can compute
$${\rm End}_{A\otimes A^{\rm op}}(M(V_*))\simeq\bigoplus_{i\in Q_0}{\rm End}_k(V_i)$$
(even without the assumption on $Q$) and
$${\rm End}_{A\otimes A^{\rm op}}(N(V_*))\simeq\bigoplus_{\alpha:i\rightarrow j}{\rm End}_k(V_i)$$
(again using the assumption on $Q$).

Moreover, let us note that 
$${\rm\bf dim}\, M(V_*)-{\rm\bf dim}\, N(V_*)={\bf f}.$$
Namely, we have already proved that
$$M(V_*)_{i,j}=\bigoplus_{j\stackrel{\omega'}{\leadsto}p\stackrel{\omega}{\leadsto}i}V_p,$$
and similarly, we find
$$N(V_*)_{i,j}=\bigoplus_{j\stackrel{\omega'}{\leadsto}p\stackrel{\alpha}{\rightarrow}q\stackrel{\omega}{\leadsto}i}V_p.$$
The only summand of $M(V_*)_{i,j}$ not appearing as a summand of $N(V_*)_{i,j}$ is the summand $V_i$ corresponding to $j\stackrel{\omega'}{\leadsto}i\stackrel{e_i}{\leadsto}i$ in case there exists a path $\omega'$ from $j$ to $i$ (again using the assumption of $Q$ admitting no parallel paths). This proves the identity of dimension vectors.

Inside $X={\rm Gr}^{{\bf f}(V_*)}_{A\otimes A^{\rm op}}(M(V_*))$, we thus have the locally closed subset consisting of quotients whose kernel is isomorphic to $N(V_*)$. This subset is isomorphic to the quotient of the set
$${\rm Hom}^{\rm inj}_{A\otimes A^{\rm op}}(N(V_*),M(V_*))$$
of injective maps by the group
$${\rm Aut}_{A\otimes A^{\rm op}}(N(V_*)).$$
By the above calculation, this is nothing else than the quotient of
$$Z=\bigoplus_{\alpha:i\rightarrow j}{\rm Hom}_k(V_i, V_i\oplus V_j)$$
by the action of
$$H=\prod_{\alpha:i\rightarrow j}{\rm GL}(V_i).$$
On the open subset $Z^0$ of $Z$ consisting of tuples 
$$(F_\alpha=\left[F_\alpha^1\atop F_\alpha^2\right]:V_i\rightarrow V_i\oplus V_j)_{\alpha:i\rightarrow j}$$ 
where all $F^1_\alpha$ are invertible, we can use the action of $H$ to normalize all 
$F^1_\alpha$ to identity maps.  This shows that $Z^0/H\simeq R_{\bf d}(Q)$. 
The action of ${\rm Aut}_{A\otimes A^{\rm op}}(M(V_*))\simeq G_{\bf d}$ on $X$ 
then restricts to the standard change of base action of $G_{\bf d}$ on 
$$R_{\bf d}\simeq Z^0/H\subset Z/H\subset X,$$
yielding the claimed equivariant embedding in to $X$, which is of course a projective variety.
\end{proof}

\section{Smoothness}\label{sec:smoothness}
The goal of this section is to prove smoothness of the Grassmannians of sub-bimodules,
generalizing the corresponding results from \cite{L25}. 
As above, we assume that the quiver $Q$ has no parallel paths. 

\begin{thm}\label{thm:smoothness}
The varieties ${\rm Gr}^{{\bf f}(V_*)}_{A\otimes A^{\rm op}}(M(V_*))$ are smooth.
\end{thm}
\begin{proof}
We use a slight generalization of \cite[Proposition 7.1]{CFR13}
(see also \cite{CFR14}):  if $B$ is an algebra of global dimension at 
most two and $M$ is a representation of $B$ such that ${\rm Ext}^1_B(M,M)=0$ and 
\[
{\rm Ext}^1(N,M/N)=0={\rm Ext}^2(N,M/N)
\]
for all $B$-subrepresentations $N$ of $M$ of dimension vector 
${\bf e}$, then ${\rm Gr}_{\bf e}^B(M)$ is smooth and irreducible of dimension 
$\langle{\bf e},{\rm\bf dim}\, M-{\bf e}\rangle$. A direct inspection of the proof of \cite[Proposition 7.1]{CFR13} immediately shows that the proof carries over to the present assumptions.

The algebra $A\otimes A^{\rm op}$ is of global dimension at most two as it is the product of two hereditary algebras. Alternatively, we can easily construct projective resolutions of the simple representations $S_i\otimes S_j$ as follows:
$$0\rightarrow\bigoplus_{{i\rightarrow p}\atop{q\rightarrow j}}P_p\otimes P_q\rightarrow\bigoplus_{i\rightarrow p}P_p\otimes P_j\oplus\bigoplus_{q\rightarrow j}P_i\otimes P_q\rightarrow P_i\otimes P_j\rightarrow S_i\otimes S_j\rightarrow 0.$$
Obviously, $P_i\otimes P_j$ has vanishing ${\rm Ext}^1$ since it is a projective bimodule.

%We also need information on injective coresolutions of the projective bimodules $P_i\otimes P_j$, and this is specific to the linearly oriented type $A_n$ quiver. We start with injective coresolutions of the $kQ$-representations $P_i$. These are of the form
%$$0\rightarrow P_i\rightarrow I_n\rightarrow I_{i-1}\rightarrow 0.$$
%From this, one can construct the injective coresolution of bimodules
%$$0\rightarrow P_i\otimes P_j\rightarrow I_n\otimes I_1\rightarrow I_{i-1}\otimes I_1\oplus I_m\otimes I_{j+1}\rightarrow I_{i-1}\otimes I_{j+1}\rightarrow 0.$$
Now assume we have an exact sequence
$$0\rightarrow N\rightarrow M=\bigoplus_iP_i\otimes V_i\otimes P_i\rightarrow M/N\rightarrow 0.$$
Using the long exact sequence induced by ${\rm Hom}_B(\_,M/N)$, we find
$$0={\rm Ext}_B^1(M,M/N)\rightarrow{\rm Ext}^1_B(N,M/N)\rightarrow{\rm Ext}_B^2(M/N,M/N)\rightarrow$$
$$\rightarrow 0={\rm Ext}_B^2(M,M/N)\rightarrow{\rm Ext}^2_B(N,M/N)\rightarrow{\rm Ext}^3_B(M/N,M/N)=0,$$
thus ${\rm Ext}_B^2(N,M/N)=0$ and $${\rm Ext}^1_B(N,M/N)\simeq{\rm Ext}_B^2(M/N,M/N).$$ 
The right hand side is a factor of ${\rm Ext}_B^2(M/N,M)$, thus we are finished if we can prove that 
the latter vanishes. 
%We compute this extension group using the above injective coresolution and find 
%that ${\rm Ext}_B^2(M/N,M)$ is a factor of
%$${\rm Hom}(M/N,\bigoplus_{i=1}^n(I_{i-1}\otimes I_{i+1})^{d_i})\simeq 
%\bigoplus_{i=1}^n((M/N)_{i-1,i+1}^*)^{d_i}$$
%But $M/N$ is of dimension vector ${\bf f}$, and ${\bf f}_{i,j}\not=0$ only if $i\geq j$. Thus we find ${\rm Ext}^2_B(M/N,M/N)=0$ as desired.

We can compute ${\rm Ext}^2_B(S_i\otimes S_j,M)$ using the above projective resolution of $S_i\otimes S_j$. Namely, we have a right exact sequence
$$\bigoplus_{i\rightarrow p}M_{pj}\oplus\bigoplus_{q\rightarrow j}M_{iq}\rightarrow
\bigoplus_{{i\rightarrow p}\atop{q\rightarrow j}}M_{pq}\rightarrow{\rm Ext}^2_B(S_i\otimes S_j,M)\rightarrow 0.$$
It is then easy to see that
$${\rm Ext}^2_B(S_i\otimes S_j,P_k\otimes P_k)=0$$
for all vertices $i,j,k$ of $Q$ such that there exists a path $j\leadsto i$, again using the assumption 
that there are no parallel paths. Now the proof can be finished as above: we only have to prove that 
${\rm Ext}^2_B(M/N,M)=0$. But $M/N$, being of dimension vector ${\bf f}$, has a composition series 
with subquotients $S_i\otimes S_j$ for $j\leadsto i$, whereas $M$ is a direct sum of the 
$P_k\otimes P_k$. Thus this Ext-group indeed vanishes.
\end{proof}

\section{Framed moduli spaces and motivic generating series}\label{sec:framedmoduli}

\subsection{Torus fixed points}
In this section we describe cellular decompositions of the varieties $X$. We start with
the description of the labeling of the set of cells.

Let us fix bases of all $V_i$ and denote by $D_i$, $i\in Q_0$ a set of cardinality $d_i=\dim V_i$ 
labeling the basis elements of $V_i$ (the sets $D_i$ do not intersect).
Let $T\subset G_{\bf d}$ be a torus of rank $\sum_id_i$
acting on $M(V_*)$ (recall that all maps in this representation are just inclusions of direct
summands).  Define 
$$D_{i,j}=\bigcup_{j\stackrel{\omega'}{\leadsto}p\stackrel{\omega}{\leadsto}i}D_p$$
for all $i,j\in Q_0$. We have natural inclusions maps 
$\iota_{\alpha,j}:D_{i,j}\subset D_{i',j}$ for all $\alpha:i\rightarrow i'$ and $j\in Q_0$, 
as well as 
$\iota_{i,\alpha^*}:D_{i,j}\subset D_{i,j'}$ for all $\alpha:j'\rightarrow j$ and $i\in Q_0$. 
Set ${\bf e}={\bf\rm dim}\, N(V_*)$ as above. The following lemma follows directly from 
the above definitions. 

\begin{lem}
The  $T$-fixed points in $X$ are labeled by tuples of subsets 
$(I_{i,j}\subset D_{i,j})_{i,j\in Q_0}$, $|I_{i,j}|={\bf e}_{i,j}$ ($i,j\in Q_0$) 
which are compatible with all the inclusion maps $\iota_{\alpha,j}$ and $\iota_{i,\alpha^*}$.
\end{lem}
\begin{proof}
 Each collection $I_{i,j}$ determines a torus fixed point in the Grassamann variety 
 ${\rm Gr}_{e_{i,j}}(M(V_*)_{i,j})$. The product of these points belongs to the
 Grassmannian of bimodules $X$ if and only if the collections $I_{i,j}$ satisfy
 all the conditions above. In what follows for a collection ${\bf I}$ of the 
 admissible sets $I_{i,j}$ we denote the corresponding point in $X$ by $p({\bf I})$.
\end{proof}

Recall (see Theorem \ref{thm:smoothness}) that our Grassmannians of sub-bimodules are smooth. 
Choosing a general enough one-dimensional sub-torus  one applies the Bialynicki-Birula
theorem \cite{B-B73,B-B74}  in order to derive an affine paving of $X$. The cells
of the paving are labeled by the collections ${\bf I}=(I_{i,j})_{i,j}$ 
from the above lemma. 
Recall (see section \ref{sec:sub-bimodules}) the quiver $\overline Q$ whose set of vertices consists 
of pairs $(i,j)\in Q_0\times Q_0$  such that there is a path from $j$ to $i$ in $Q$. 
The vertices $\overline Q_0$ form a poset and for each upper ideal $C$ in $\overline Q_0$
which can not be written as a disjoint union of two upper ideals, one gets an indecomposable
$\overline Q$ module $M(C)$. In particular, all the components of $M(C)$ are at most
one-dimensional and the non-trivial components are labeled by the elements of $C$.

\begin{lem}
The point $p({\bf I})$ as a $\overline Q$ module is isomorphic to the direct sum of
modules $M(C)$ over all upper ideals $C$ satisfying the following condition:
\[
\exists\ (i=1,\dots,n \ \text{ and }\ r\in D_i) \text{ such that } 
r\in I_{a,b} \Leftrightarrow (a,b)\in C.
\]
The representation $M(C)$ further decomposes into direct sum of $M(C_u)$ such that 
$C=\sqcup_u C_u$ and each $C_u$ is an indecomposable upper ideal.
\end{lem}
\begin{proof}
It is enough to note that the maps  $\iota_{\alpha,j}$ and $\iota_{i,\alpha^*}$ are 
embeddings.
\end{proof}

\subsection{Framed moduli spaces}

Using the methods of \cite{R08}, one can prove that $X$ is isomorphic to a framed moduli space.
Define vector spaces $M_{i,j}$ as $V_i$ if there exists a path from $j$ to $i$ in $Q$, and as
$0$ otherwise. Consider the variety $R_{\bf e}(Q\times Q^{\rm op},I)$ of representations of
$Q\times Q^{\rm op}$ on the vector spaces $M_{i,j}$ which satisfy all commutativity relations (making them into a representation $M$ of $kQ\otimes kQ^{\rm op}$), on which the group 
$$G_{\bf e}=\prod_{j\leadsto i}{\rm GL}(M_{i,j})$$ 
acts. We define the framed representation variety as 
$$R_{\bf e}^{{\bf d}-{\rm fr}}(Q\times Q^{\rm op},I)=R_{\bf e}(Q\times Q^{\rm op},I)\times\bigoplus_{i\in Q_0}{\rm Hom}_k(V_i,M_{i,i}),$$
in which we consider the open subset 
$R_{\bf e}^{{\bf d}-{\rm fr}}(Q\times Q^{\rm op},I)^{\rm st}$ of pairs $(M,F)$ 
such that the direct sum of the images $F_i(V_i)\subset M_{i,i}$ generates $M$ as a
representation of $Q\times Q^{\rm op}$. The group $G_{\bf e}$ acts freely on the latter, 
and the quotient
$$M_{\bf e}^{{\bf d}-{\rm fr}}(Q\times Q^{\rm op},I)=R_{\bf e}^{{\bf d}-{\rm fr}}(Q\times Q^{\rm op},I)^{\rm st}/G_{\bf e}$$
exists as quasiprojective variety since it is a GIT quotient of a set of stable points.  

\begin{prop}
The framed moduli space  $M_{\bf e}^{{\bf d}-{\rm fr}}(Q\times Q^{\rm op},I)$ is isomorphic 
to the Grassmannian of sub-bimodules $X$.
\end{prop}
\begin{proof}
It suffices to note the following: the images of the maps $F_i$ generate $M$ 
if and only if the representation $M$ is a quotient of the projective 
representation $M(V_*)$.
\end{proof}

We stratify $R_{\bf e}^{{\bf d}-{\rm fr}}(Q\times Q^{\rm op},I)$ by locally closed subsets 
$R_{\bf e}^{{\bf d}-{\rm fr}}(Q\times Q^{\rm op},I)_{\bf g}$ consisting of all pairs $(M,F)$
such that the images $F_i(V_i)$ generate a sub-bimodule of dimension vector 
${\bf g}\leq{\bf e}$. In particular, 
$$R_{\bf e}^{{\bf d}-{\rm fr}}(Q\times Q^{\rm op},I)_{\bf e}=R_{\bf e}^{{\bf d}-{\rm fr}}(Q\times Q^{\rm op},I)^{\rm st}$$ is the open stratum. We have
$$R_{\bf e}^{{\bf d}-{\rm fr}}(Q\times Q^{\rm op},I)_{\bf g}\simeq G_{\bf e}\times^{P_{{\bf g},{\bf e}-{\bf g}}}Z_{{\bf g},{\bf e}-{\bf g}},$$
where $Z_{{\bf g},{\bf e}-{\bf g}}\subset R_{\bf e}^{{\bf d}-{\rm fr}}(Q\times Q^{\rm op},I)$ consists of pairs
$$(\left[\begin{array}{cc}M'&*\\ 0&M''\end{array}\right],\left[{f'\atop 0}\right])$$
where $(M',f')\in R_{\bf g}^{{\bf d}-{\rm fr}}(Q\times Q^{\rm op},I)^{\rm st}$
and $M''\in R_{{\bf f}-{\bf g}}(Q\times Q^{\rm op},I)$. We claim that that the natural projection
$$p:Z_{{\bf g},{\bf e}-{\bf g}}\rightarrow R_{\bf g}^{{\bf d}-{\rm fr}}(Q\times Q^{\rm op},I)\times R_{{\bf e}-{\bf g}}(Q\times Q^{\rm op},I)$$
is a Zariski-locally trivial vector bundle of rank $\langle{\bf e}-{\bf g},{\bf g}\rangle_B$,
where again $B=kQ\otimes kQ^{\rm op}$, and the homological Euler form 
$\langle\_,\_\rangle_B$ can be computed using the projective resolution of simple bimodules as
$$\langle{\bf x},{\bf y}\rangle=\sum_{i,j}{x}_{ij}{y}_{ij}-\sum_{{i\rightarrow i'}\atop j}x_{i,j}y_{i',j}-\sum_{i\atop{j'\rightarrow j}}x_{i,j}y_{i,j'}+\sum_{{i\rightarrow i'}\atop{j'\rightarrow j}}x_{ij}y_{i'j'}.$$
Namely, constancy of the fiber dimensions of $p$ is a consequence of the 
${\rm Ext}^2_B$-vanishing property
$${\rm Ext}_B^2(S_i\otimes S_j,P_k\otimes P_k)=0$$
whenever $j\leadsto i$.

The above stratification immediately allows us to formulate a recursive formula for motives. We work in the localization $K_0({\rm Var}_{k})[\mathbb{L}^{-1},(1-\mathbb{L}^i)^{-1},\, i\geq 1]$ of the Grothendieck ring of $k$-varieties, for $\mathbb{L}=[\mathbb{A}^1]$ the Lefschetz motive, $[Z]$ denoting the class of a variety $Z$.

\begin{thm}\label{thm:recursion}
One has
\begin{multline*}
\mathbb{L}^{\sum_{i\in Q_0}d_i^2}\cdot\frac{[R_{\bf e}(Q\times Q^{\rm op},I)]}{[G_{\bf e}]}=\\
\sum_{{\bf g}\leq{\bf e}}[M_{\bf g}^{{\bf d}-{\rm fr}}(Q\times Q^{\rm op},I)]\cdot\frac{[R_{{\bf e}-{\bf g}}(Q\times Q^{\rm op},I)]}{[G_{{\bf e}-{\bf g}}]}\cdot\mathbb{L}^{\langle{\bf e}-{\bf g},{\bf g}\rangle_B}.
\end{multline*}
\end{thm}
\begin{proof}
The proof follows from the stratification described above.
\end{proof}

We note that thanks to the existence of cellular decomposition,
the framed moduli spaces are trivially motivated, that is, their motive is a polynomial in the Lefschetz motive, and consequently,
$$[M_{\bf g}^{{\bf d}-{\rm fr}}(Q\times Q^{\rm op},I)]=\sum_i\dim H^i(M_{\bf g}^{{\bf d}-{\rm fr}}(Q\times Q^{\rm op},I),\mathbb{Q})\cdot\mathbb{L}^{i/2}.$$

\section{Quivers with parallel paths}\label{sec:parallel}
In this section we drop the condition that $Q$ has no parallel paths and develop
a more general construction. We show that in general the Grassmannians of sub-bimodules are more
complicated; in particular, we give
a non-smooth example.

Let $Q$ be a quiver with no cycles (but parallel paths are allowed). Let $\overline Q$ be a 
new quiver whose vertices are paths $\omega$ in $Q$ and there are two types of arrows.
Namely, let $\omega$ be a path in $Q$ with the source $s(\omega)\in Q_0$ and target $t(\omega)\in Q_0$. Let $\al,\beta\in Q_1$ be the arrows satisfying $t(\al)=s(\omega)$,
$s(\beta)=t(\omega)$. Then the following are the arrows in $\overline Q$:
\begin{gather}\label{eq:firsttype}
\left(s(\omega) \stackrel{\omega}{\rightsquigarrow} t(\omega)\right) 
\stackrel{\widehat\alpha(\omega)}{\longrightarrow} 
\left(s(\alpha)\stackrel{\alpha}{\rightarrow} t(\al) = s(\omega)  \stackrel{\omega}{\rightsquigarrow} t(\omega)\right),\\ \label{eq:secondtype}
\left(s(\omega) \stackrel{\omega}{\rightsquigarrow} t(\omega)\right) 
\stackrel{\widehat\beta(\omega)}{\longrightarrow} 
\left(s(\omega)  \stackrel{\omega}{\rightsquigarrow} t(\omega)=s(\beta) \stackrel{\beta}{\rightarrow} t(\beta)\right).
\end{gather}
We define the ideal $I$ in the path algebra of $\overline Q$ generated by the relations
$\widehat\alpha(\omega)\widehat\beta(\omega)=\widehat\beta(\omega)\widehat\alpha(\omega)$ for
all $\alpha,\beta\in Q_1$, $\omega\in\overline Q_0$, satisfying the conditions as above. 

Let us fix a collection of vector spaces $V_i$, $i\in Q_0$ with $\dim V_i=d_i$. 
We define a  ${k}\overline Q/I$ module $M(V_*)$ as follows. 
For a path $\omega$ (a vertex of $\overline Q$) the space $M(V_*)_\omega$ is defined as
\[
M(V_*)_\omega = \bigoplus V_p, \ \omega \text{ passes through } p;
\]
the maps $M(V_*)_{\widehat\alpha(\omega)}$ and $M(V_*)_{\widehat\beta(\omega)}$ 
are the obvious embeddings. Also, given a dimension vector $\dim (V_*)=(d_i)_{i\in Q_0}$ we define the dimension vectors
${\bf f}(\dim V_*)={\bf f} = (f_\omega)_{\omega\in \overline Q_0}$, 
${\bf e}(\dim V_*)={\bf e} = (e_\omega)_{\omega\in \overline Q_0}$  by the formulas 
\[
f_\omega = d_{t(\omega)},\quad e_\omega =  
\sum_{\substack{p\in\omega\\ p\ne t(\omega)}} d_p. 
\]
In particular, $\dim M(V_*) = {\bf f} + {\bf e}$. We are interested in the Grassmannian of 
submodules
\begin{equation}\label{eq:Grassmannian}
X = \mathrm{Gr}_{{\bf e}(\dim V_*)}^{\overline Q/I} (M(V_*)).
\end{equation}

Given a $Q$ module $V$ we denote by $\overline V$ the ${k}\overline Q/I$ module (of dimension
${\bf e}(\dim V)$) defined as follows.
For a lazy (length zero) path  $\omega$ in $Q$ the space $\overline V_\omega$ is trivial (zero). 
For a non-lazy path one defines
\[
\overline V_\omega = \bigoplus_{Q_1\ni \alpha\in \omega} \{x+V_\alpha(x):\ x\in s(\alpha)\} \subset \bigoplus_{Q_0\ni p\in\omega} V_i 
\]
(here the sums are taken over the arrows of $Q$ which show up in the path $\omega$).
The maps $\overline V_{\widehat\alpha(\omega)}$ and $\overline V_{\widehat\beta(\omega)}$ 
are the obvious embeddings. 

\begin{lem}
The $\mathbf{k}\overline Q/I$ module $\overline V$ is a subrepresentation of $\widehat M(V_*)$ of dimension 
${\bf e}(\dim V)$. 
\end{lem}
\begin{proof}
The sum $\bigoplus_{Q_1\ni \alpha\in \omega} \{x+V_\alpha(x):\ x\in s(\alpha)\}$ is direct thanks
to the acyclicity of $Q$.
\end{proof}

This Lemma allows to embed the representation variety $\mathrm{Rep}_{{\bf d}}(\mathbf{k}\overline Q/I)$ into the Grassmannian $X$ \eqref{eq:Grassmannian}. 
\begin{conj}
The closure of the image of the above embedding coincides with $X$. In particular, $X$ is irreducible.
\end{conj}

Let us consider the example $Q=1 \rightrightarrows 2 \rightarrow 3$.
Then $\overline Q$ is of the form 

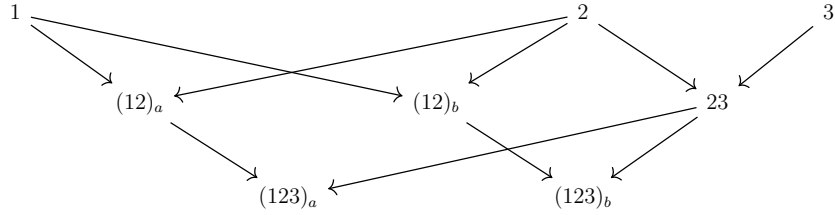
\[
\begin{tikzcd}[scale cd = 0.7]
1\arrow[dr]\arrow[drrr] & & & & 2\arrow[dlll]\arrow[dl]\arrow[dr] & & 3\arrow[dl]  \\
& (12)_a \arrow[dr] & & (12)_b \arrow[dr] & & 23 \arrow[dl]\arrow[dlll]&   \\
&  & (123)_a & & (123)_b & & 
\end{tikzcd}	
\]
(here the subscripts $a$ and $b$ correspond to the two arrows between the vertices $1$ 
and $2$).
The ideal $I$ is generated by the two commutative squares in the picture. 

Given a collection of vector spaces $V_1,V_2,V_3$ of dimensions $d_1,d_2,d_3$ the module $\widehat M(V_*)$ is given by

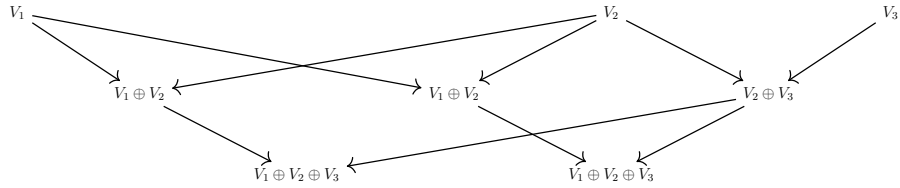
\[
\begin{tikzcd}[scale cd = 0.5]
V_1\arrow[dr]\arrow[drrr] & & & & V_2\arrow[dlll]\arrow[dl]\arrow[dr] & & V_3\arrow[dl]  \\
& V_1\oplus V_2 \arrow[dr] & & V_1\oplus V_2 \arrow[dr] & & V_2\oplus V_3 \arrow[dl]\arrow[dlll]&   \\
&  & V_1\oplus V_2\oplus V_3 & & V_1\oplus V_2\oplus V_3 & & 
\end{tikzcd}	
\]

\medskip

The dimension vector ${\bf e}(V)$ is given by  

\medskip

\[
\begin{tikzcd}[scale cd = 0.9]
0& & & & 0  & & 0  \\
& d_1  & & d_1  & & d_2&   \\
&  & d_1+d_2 & & d_1+d_2 & & 
\end{tikzcd}	
\]

\medskip

The corresponding Grassmannian is irreducible.
To see this one notices that we can forget the vertices of $\overline Q$ in the upper row 
(since the values of the dimension vector ${\bf e}$ are anyway zero); the remaining 
quiver is of type $A_5$:
\[
(12)_a \rightarrow (123)_a \leftarrow 23 \rightarrow (123)_b \leftarrow (12)_b.
\]
The module $M(V_*)$ takes the form
\[
U_{1,2}^{\oplus d_1}\oplus U_{1,5}^{\oplus d_2}\oplus U_{2,4}^{\oplus d_3}\oplus U_{4,5}^{\oplus d_1},
\]
where $U_{i,j}$ are indecomposables supported between vertices $i$ and $j$. The dimension vector
${\bf e}$ reads as $(d_1,d_1+d_2,d_2,d_1+d_2,d_1)$. In particular,
\[
{\bf e} = \dim U_{1,2}^{\oplus d_1}\oplus  U_{2,4}^{\oplus d_2}\oplus U_{4,5}^{\oplus d_1},
\]
Now to prove the irreducibility it suffices to observe the $\mathrm{Ext}$ 
vanishing criterion from \cite{CFR12}.

Finally, let us consider the case $\dim V_*=(1,1,1)$. Then one checks that the 
Poincar\'e polynomial of the quiver Grassmannian $X$   is equal to
 $1+5q+6q^2+q^3$ and hence $X$ is singular.

\end{document}